\documentclass[a4paper,english,fontsize=11pt,parskip=half,abstracton]{scrartcl}
\usepackage{babel}
\usepackage[utf8]{inputenc}
\usepackage[T1]{fontenc}
\usepackage[a4paper,left=25mm,right=25mm,top=30mm,bottom=30mm]{geometry}
\usepackage{amsmath}
\usepackage{amsthm}
\usepackage{amssymb}
\usepackage{enumerate}
\usepackage{aliascnt}
\usepackage[bookmarks=false,
            pdftitle={Pseudo Sylow numbers},
            pdfauthor={Benjamin Sambale},
            pdfkeywords={Sylow's theorem},
            pdfstartview={FitH}]{hyperref}

\newtheorem*{ThmA}{Theorem A}

\newtheorem{Thm}{Theorem} 
\newaliascnt{Lem}{Thm}
\newtheorem{Lem}[Lem]{Lemma}
\aliascntresetthe{Lem}
\newaliascnt{Prop}{Thm}
\newtheorem{Prop}[Prop]{Proposition}
\aliascntresetthe{Prop}
\newaliascnt{Cor}{Thm}

\aliascntresetthe{Cor}
\theoremstyle{definition}
\newaliascnt{Def}{Thm}

\aliascntresetthe{Def}
\newaliascnt{Ex}{Thm}

\aliascntresetthe{Ex}

\numberwithin{equation}{section}

\allowdisplaybreaks[1]

\renewcommand{\phi}{\varphi}
\newcommand{\C}{\operatorname{C}}
\newcommand{\N}{\operatorname{N}}

\newcommand{\ZZ}{\mathbb{Z}}

\newcommand{\Aut}{\operatorname{Aut}}

\newcommand{\Sym}{\operatorname{Sym}}
\newcommand{\pcore}{\operatorname{O}}

\newcommand{\Syl}{\operatorname{Syl}}

\newcommand{\Ker}{\operatorname{Ker}}

\mathchardef\ordinarycolon\mathcode`\:  
 \mathcode`\:=\string"8000
 \begingroup \catcode`\:=\active
   \gdef:{\mathrel{\mathop\ordinarycolon}}
 \endgroup

\title{Pseudo Sylow numbers}
\author{Benjamin Sambale\footnote{Fachbereich Mathematik, TU Kaiserslautern, 67653 Kaiserslautern, Germany, 
\href{mailto:sambale@mathematik.uni-kl.de}{sambale@mathematik.uni-kl.de}}}
\date{\today}

\begin{document}
\frenchspacing
\maketitle
\begin{abstract}\noindent
One part of Sylow's famous theorem in group theory states that the number of Sylow $p$-subgroups of a finite group is always congruent to $1$ modulo $p$. Conversely, Marshall Hall has shown that not every positive integer $n\equiv 1\pmod{p}$ occurs as the number of Sylow $p$-subgroups of some finite group. While Hall's proof relies on deep knowledge of modular representation theory, we show by elementary means that no finite group has exactly $35$ Sylow $17$-subgroups.
\end{abstract}

\section{Introduction}

Every student of abstract algebra encounters at some point one of the most fundamental theorems on finite groups: 

\begin{Thm}[\textsc{Sylow}]\label{sylow}
Let $G$ be a group of finite order $n=p^am$ where $p$ is a prime not dividing $m$. 
Then the number of subgroups of $G$ of order $p^a$ is congruent to $1$ modulo $p$. In particular, there is at least one such subgroup.
\end{Thm}

The subgroups described in \autoref{sylow} are called \emph{Sylow $p$-subgroups} of $G$.
Apart from Sylow's original proof~\cite{Sylow} from 1872, a number of different proofs appeared in the literature and they are presented in the survey article \cite{Waterhouse}. More recently, a very elementary proof by Robinson~\cite{RobinsonSylow} of the last part of \autoref{sylow} appeared in \textsc{Monthly}.

Fixing $p$, it is natural to ask if every positive integer $n\equiv 1\pmod{p}$ is a \emph{Sylow $p$-number}, i.\,e., $n$ is the number of Sylow $p$-subgroups of some finite group. 
Certainly, $n=1$ is a Sylow $p$-number for the trivial group $G=\{1\}$ and every prime $p$. Moreover, every odd $n$ is a Sylow $2$-number for the dihedral group of order $2n$. This is the symmetry group of the regular $n$-gon and the Sylow $2$-subgroups are in one-to-one correspondence with the reflections. For odd primes $p$ the question is more delicate.

Philip Hall~\cite{Hallgroups} observed that in \emph{solvable} groups the prime factorization of a Sylow $p$-number $n=p_1^{a_1}\cdots p_s^{a_s}$ satisfies $p_i^{a_i}\equiv 1\pmod{p}$ for $i=1,\ldots,s$. For example, no solvable groups has exactly six Sylow $5$-subgroups. Nevertheless, the symmetry group of the dodecahedron of order $120$ does have six Sylow $5$-subgroups which can be identified with the $5$-fold rotations of the six axes. 
About forty years later Marshall Hall~\cite{MHall} reduced the determination of the Sylow $p$-numbers to \emph{simple} groups. (Recall that simple groups are like prime numbers in that they have only two normal subgroups: the trivial group and the whole group.) More precisely, he showed that every Sylow $p$-number is a product of prime powers $q^t\equiv 1\pmod{p}$ and Sylow $p$-numbers of (nonabelian) simple groups.
Conversely, every such product is in fact a Sylow $p$-number which can be seen by taking suitable direct products of affine groups and simple groups. 

Since nowadays the extremely complicated classification of the finite simple groups is believed to be complete (see \cite{CFSG}), one can in principle determine the Sylow $p$-numbers by going through the list of simple groups (see \cite{OEIS}). M. Hall instead used Brauer's sophisticated theory of $p$-blocks of defect $1$ (a part of modular representation theory) to show that \emph{not} every positive integer $n\equiv 1\pmod{p}$ is a Sylow $p$-number. More precisely, he constructed such \emph{pseudo} Sylow $p$-numbers for every odd prime $p$ (for instance $n=22$ works for $p\in\{3,7\}$). In the present paper we are content to provide only one such number which is a special case of \cite[Theorem~3.1]{MHall}:

\begin{ThmA}\label{17}
No finite group has exactly $35$ Sylow $17$-subgroups.
\end{ThmA}

The Fermat prime $17$ is chosen to make the proof as easy as possible. Apart from Sylow's theorem we only use first principles of group actions. It seems that such an elementary proof has not appeared in the literature so far.

Lastly, we remark that Frobenius~\cite{FrobeniusSylow2} has extended Sylow's theorem to the following: If a prime power $p^a$ divides the order of a finite group $G$, then the number $n_{p^a}$ of subgroups of order $p^a$ in $G$ is congruent to $1$ modulo $p$. Moreover, if $p^{a+1}$ divides $|G|$, it is known by work of P. Hall~\cite[Lemma~4.61 and Theorem~4.6]{HallFrobenius} that $n_{p^a}$ is congruent to $1$ or $1+p$ modulo $p^2$. In particular, the number of $17$-subgroups of a fixed order of any finite group is never $35$. 
This gives rise to \emph{pseudo Frobenius numbers} which are those positive integers $n$ congruent to $1$ or $1+p$ modulo $p^2$ such that no finite group has exactly $n$ subgroups of order $p^a$ for some $a\ge 0$. The existence question of pseudo Frobenius numbers will be resolved in a different paper (see \cite{SambaleFrob}).

\section{Proof of Theorem~A}

We assume that the reader is familiar with elementary group theory as it is given for example in \cite[Chapter I]{Lang}. In order to introduce notation we review a few basic facts.

In the following $G$ is always a finite group with identity $1$. 
Let $\Omega$ be a finite nonempty set. Then the permutations of $\Omega$ form the \emph{symmetric group} $\Sym(\Omega)$ with respect to the composition of maps. Let $S_n:=\Sym(\{1,\ldots,n\})$ be the symmetric group of \emph{degree} $n$. The even permutations in $S_n$ form the \emph{alternating group} $A_n$ of degree $n$. Recall that $|S_n|=n!$ and $|S_n:A_n|=2$ for $n\ge 2$.

An \emph{action} of $G$ on $\Omega$ is a map 
\begin{align*}
G\times\Omega&\to\Omega,\\
(g,\omega)&\mapsto {^g\omega}
\end{align*}
such that $^1\omega=\omega$ and $^{gh}\omega={^g({^h\omega})}$ for all $\omega\in\Omega$ and $g,h\in G$. 
Every action determines a group homomorphism $\sigma:G\to\Sym(\Omega)$ which sends $g\in G$ to the permutation $\omega\mapsto{^g\omega}$ of $\Omega$.
We call $\Ker(\sigma)$ the \emph{kernel} of the action. If $\Ker(\sigma)=1$, we say that $G$ acts \emph{faithfully} on $\Omega$. 
For $\omega\in\Omega$, the set $^G\omega:=\{{^g\omega}:g\in G\}$ is called the \emph{orbit} of $\omega$ under $G$. Finally, the \emph{stabilizer} of $\omega$ in $G$ is given by $G_\omega:=\{g\in G:{^g\omega}=\omega\}\le G$. 

\begin{Prop}[Orbit-stabilizer theorem]\label{orbstab}
For $g\in G$ and $\omega\in\Omega$ we have 
\[|^G\omega|=|G:G_\omega|.\]
\end{Prop}
\begin{proof}
It is easy to check that the map $G/G_\omega\to {^G\omega}$, $gG_\omega\mapsto{^g\omega}$ is a well-defined bijection (see \cite[Proposition~I.5.1]{Lang}).
\end{proof}

The most relevant action in the situation of Sylow's theorem is the action of $G$ on itself by \emph{conjugation}, i.\,e., $^gx:=gxg^{-1}$ for $g,x\in G$. Then the orbits are called \emph{conjugacy classes} and the stabilizer of $x$ is the \emph{centralizer} $\C_G(x):=\{g\in G:gx=xg\}$. Conjugation also induces an action of $G$ on the set of subgroups of $G$. Here the stabilizer of $H\le G$ is the \emph{normalizer} $\N_G(H):=\{g\in G:gH=Hg\}$. 
Clearly, $\N_G(H)$ acts by conjugation on $H$ and the corresponding kernel is 
\[\C_G(H):=\{g\in G:gh=hg\ \forall h\in H\}\unlhd\N_G(H).\]
The action on the set of subgroups can be restricted onto the set $\Syl_p(G)$ of Sylow $p$-subgroups, since conjugation preserves order. Then the following supplement to Sylow's theorem implies that this action of $G$ has only one orbit on $\Syl_p(G)$.

\begin{Prop}[\textsc{Sylow}'s second theorem]\label{sylow2}
Let $P\in\Syl_p(G)$. Then every $p$-subgroup of $G$ is conjugate to a subgroup of $P$. In particular, all Sylow $p$-subgroups of $G$ are conjugate. 
\end{Prop}
\begin{proof}
See \cite[Theorem~I.6.4]{Lang}.
\end{proof}

The Propositions~\ref{orbstab} and \ref{sylow2} imply that $\lvert\Syl_p(G)\rvert=|G:\N_G(P)|$ for any $P\in\Syl_p(G)$. Hence, by Lagrange's theorem, the number of Sylow $p$-subgroups of $G$ divides $|G|$ (see \cite[Proposition~I.2.2]{Lang}). Moreover, the \emph{$p$-core}
\[\pcore_p(G):=\bigcap_{P\in\Syl_p(G)}P\]
of $G$ lies in the kernel of the conjugation action of $G$ on $\Syl_p(G)$.

Our next ingredient is a less known result by Brodkey~\cite{Brodkey}. 

\begin{Prop}[\textsc{Brodkey}]\label{brodkey}
Suppose that $G$ has abelian Sylow $p$-subgroups. Then there exist $P,Q\in\Syl_p(G)$ such that $P\cap Q=\pcore_p(G)$.
\end{Prop}
\begin{proof}
Choose $P,Q\in\Syl_p(G)$ such that $|P\cap Q|$ is as small as possible. Since $P$ and $Q$ are abelian, it follows that $P\cap Q\unlhd P$ and $P\cap Q\unlhd Q$. This means that $P$ and $Q$ are Sylow $p$-subgroups of $N:=\N_G(P\cap Q)$. Now let $S\in\Syl_p(G)$ be arbitrary. By \autoref{sylow2}, there exists $g\in N$ such that $^gS\cap N={^g(S\cap N)}\le P$. We conclude that
\[^gS\cap Q={^gS}\cap N\cap Q\le P\cap Q.\]
By the choice of $P$ and $Q$, we have equality $P\cap Q={^gS}\cap Q\le{^gS}$. Conjugating by $g^{-1}$ on both sides yields $P\cap Q={^{g^{-1}}(P\cap Q)}\le S$. Since $S$ was arbitrary, we obtain $P\cap Q\le\pcore_p(G)\le P\cap Q$ as desired.
\end{proof}

For the proof of Theorem~A we need three more specific lemmas.

\begin{Lem}\label{centalt}
Let $p$ be an odd prime and let $\sigma$ be a product of two disjoint $p$-cycles in $A_{2p}$. Then \[\lvert\C_{A_{2p}}(\sigma)\rvert=p^2.\]
\end{Lem}
\begin{proof}
Although the claim can be proved with the orbit-stabilizer theorem, we prefer a more direct argument.
First observe that $\sigma$ is in fact an even permutation and therefore lies in $A_{2p}$.
Without loss of generality, we may assume that $\sigma=(1,\ldots,p)(p+1,\ldots,2p)$. 
Then $\langle(1,\ldots,p),(p+1,\ldots,2p)\rangle\le\C_{A_{2p}}(\sigma)$ and we obtain $\lvert\C_{A_{2p}}(\sigma)\rvert\ge p^2$. 

For the converse inequality, let $\tau\in\C_{S_{2p}}(\sigma)$. There are (at most) $2p$ choices for $\tau(1)$. For $i=2,\ldots,p$ we have $\tau(i)=\tau(\sigma^{i-1}(1))=\sigma^{i-1}(\tau(1))$. Thus after $\tau(1)$ is fixed, there are only $p$ possibilities for $\tau(p+1)$ left. Again $\tau(p+i)=\sigma^{i-1}(\tau(p+1))$ for $i=2,\ldots,p$. Altogether there are at most $2p^2$ choices for $\tau$ and we obtain $\lvert\C_{S_{2p}}(\sigma)\rvert\le 2p^2$. Observe that \[\tau:=(1,p+1)(2,p+2)\ldots(p,2p)\in\C_{S_{2p}}(\sigma),\] 
but since $p$ is odd we have $\tau\notin A_n$. Hence, $\C_{A_{2p}}(\sigma)\subsetneq\C_{S_{2p}}(\sigma)$ and Lagrange's theorem yields $\lvert\C_{A_{2p}}(\sigma)\rvert\le\lvert\C_{S_{2p}}(\sigma)\rvert/2\le p^2$.
\end{proof}

\begin{Lem}\label{NC}
Assume that $G$ has a Sylow $p$-subgroup $P$ of order $p$. 
Then $\N_G(P)/\C_G(P)$ is cyclic of order dividing $p-1$.
\end{Lem}
\begin{proof}
As we have remarked after \autoref{orbstab}, $\N_G(P)$ acts by conjugation on $P$ with kernel $\C_G(P)$. By the first isomorphism theorem, $\N_G(P)/\C_G(P)$ is isomorphic to a subgroup of $\Sym(P)$. Since conjugation induces automorphisms on $P$, we may even regard $\N_G(P)/\C_G(P)$ as a subgroup of the automorphism group $\Aut(P)$. By hypothesis, $P\cong\ZZ/p\ZZ$ and we obtain $\Aut(P)\cong(\ZZ/p\ZZ)^\times$. Now a standard fact in algebra states that $(\ZZ/p\ZZ)^\times$ is cyclic of order $p-1$ (see \cite[Theorem~IV.1.9]{Lang}). The claim follows with Lagrange's theorem.
\end{proof}

\begin{Lem}\label{cyc2}
If $G$ has a cyclic Sylow $2$-subgroup $P$, then there exists a unique $N\unlhd G$ such that $|G:N|=|P|$.
\end{Lem}
\begin{proof}
Since this is a common exercise in many textbooks (see \cite[Exercise~6.10]{IsaacsAlgebra} for instance), we only sketch the proof. 
We argue by induction on $|P|=2^n$. For $n=0$ the claim holds with $N=G$. Thus, let $n\ge 1$.
It is easy to see that $G$ acts faithfully on itself by multiplication on the left, i.\,e., $^gx:=gx$ for $g,x\in G$. Hence, we may regard $G$ as a subgroup of $S_{|G|}$. Doing so, every nontrivial element of $G$ is a permutation without fixed points.
Let $x$ be a generator of $P$. Then $x$ is a product of $|G|/2^n$ disjoint $2^n$-cycles. In particular, $x$ is an odd permutation and $H:=G\cap A_{|G|}$ is a normal subgroup of $G$ of index $2$. Moreover, $P\cap H$ is a cyclic Sylow $2$-subgroup of $H$. By induction there exists a unique $N\unlhd H$ with $|H:N|=2^{n-1}$. For $g\in G$ we have $gNg^{-1}\unlhd gHg^{-1}=H$ and $|H:gNg^{-1}|=|H:N|$. The uniqueness of $N$ shows that $N=gNg^{-1}\unlhd G$ and \[|G:N|=|G:H||H:N|=2^n.\] 
Finally, if $M\unlhd G$ with $|G:M|=2^n$, then $M\unlhd H$ and the uniqueness of $N$ gives $M=N$.
\end{proof}

\begin{proof}[Proof of Theorem~A]
Let $G$ be a minimal counterexample. 
For ease of notation let $p=17$ and $n=2p+1=35$.
 
\textbf{Step 1:} $G$ acts faithfully on $\Syl_p(G)$ and $G\le A_n$.\\
Let $K\unlhd G$ be the kernel of the conjugation action of $G$ on $\Syl_p(G)$. We show that 
\begin{align*}
\gamma:\Syl_p(G)&\to\Syl_p(G/K),\\
P&\mapsto PK/K
\end{align*}
is a bijection. For $P\in\Syl_p(G)$, the isomorphism theorems show that $PK/K\cong P/P\cap K$ is a $p$-group, and $|G/K:PK/K|=|G:PK|$ divides $|G:P|$ and thus is not divisible by $p$ (see \cite[p. 17]{Lang}). Hence, $PK/K$ is a Sylow $p$-subgroup of $G/K$. By \autoref{sylow2}, every Sylow $p$-subgroup of $G/K$ has the form $(gK)PK/K(gK)^{-1}=gPg^{-1}K/K$ for some $g\in G$. Hence, $\gamma$ is surjective. To show injectivity, let $P,Q\in\Syl_p(G)$ such that $PK/K=QK/K$. Then $PK=QK$. Since $K$ acts trivially on $\Syl_p(G)$, $P$ is the only Sylow $p$-subgroup of $PK$ and $Q$ is the only Sylow $p$-subgroup of $QK$. Hence, $P=Q$ and $\gamma$ is injective. 

It follows that $G/K$ has exactly $n$ Sylow $p$-subgroups and by the choice of $G$ we must have $K=1$. Therefore, $G$ acts faithfully on $\Syl_p(G)$ and we may regard $G$ as a subgroup of $S_n$. Since $A_n$ contains every element of odd order, every Sylow $p$-subgroup of $G$ lies in $A_n$. Consequently, $G\cap A_n$ is also a counterexample and we obtain $G\le A_n$ by minimality of $G$.

In the following we fix $P\in\Syl_p(G)$.

\textbf{Step 2:} $|P|=p$.\\
Since $|S_n|=n!$ is not divisible by $p^3$, Step 1 and Lagrange's theorem already imply $|P|\le p^2$. In particular, $P$ is abelian (see \cite[Exercise~I.24]{Lang}). Since $\pcore_p(G)$ lies in the kernel of the conjugation action on $\Syl_p(G)$, Step~1 also yields $\pcore_p(G)=1$. By Brodkey's proposition there exists $Q\in\Syl_p(G)$ such that $P\cap Q=1$. Since $Q$ is the only Sylow $p$-subgroup of $\N_G(Q)$, we conclude that $\N_P(Q)\le P\cap Q=1$. The orbit-stabilizer theorem applied to the conjugation action of $P$ on $\Syl_p(G)$ yields 
\[|P|=|P:\N_P(Q)|=|^PQ|\le n<p^2.\] 
Hence, $|P|=p$.

\textbf{Step 3:} $|G|=5\cdot 7\cdot 17$.\\
By construction, $\N_G(P)$ is the stabilizer of $P$ and Step~1 implies 
\[P\le\C_G(P)\le\N_G(P)\le S_{n-1}\cap A_n=A_{2p}.\]
It follows easily from Step~2 that $P$ has orbits of size $1$, $p$ and $p$ on $\Syl_p(G)$. Hence, $P$ is generated by a product of two disjoint $p$-cycles in $A_{2p}$. Now it follows from Step~1 and \autoref{centalt} that 
\[P\le \C_G(P)=\C_{A_{2p}}(P)\cap G=P.\] 
Consequently, $\N_G(P)/P=\N_G(P)/\C_G(P)$ is cyclic of order dividing $p-1=2^4$ by \autoref{NC}. Since $|G:\N_G(P)|=n$ is odd, $\N_G(P)$ contains a Sylow $2$-subgroup of $G$. In particular, the Sylow $2$-subgroups of $G$ are cyclic and \autoref{cyc2} yields a normal subgroup $N\unlhd G$ of order $|P||G:\N_G(P)|=5\cdot 7\cdot 17$. Since every Sylow $p$-subgroup of $G$ is contained in $N$, we have $G=N$ by minimality of $G$. 

\textbf{Step 4:} Contradiction.\\
By Sylow's theorem, $G$ has a unique Sylow $5$-subgroup $T\unlhd G$. Then $PT$ is a subgroup of $G$ of order $5\cdot 17$. Again by Sylow's theorem, $PT$ has only one Sylow $p$-subgroup. In particular $P\unlhd PT$ and $T\le\N_G(P)$. This gives the contradiction $|G:\N_G(P)|<n$. 
\end{proof}

It is possible to modify the proof above to construct more pseudo Sylow numbers. 
In fact the first three steps work more generally whenever $\lvert\Syl_p(G)\rvert=2p+1$. One ends up with a group of odd order which must be solvable according to the celebrated Feit--Thompson theorem~\cite{FeitThompson}. Then the factorization property by P. Hall mentioned in the introduction implies that $2p+1$ is a prime power. Unfortunately, the long proof of the Feit--Thompson theorem is even more challenging than the methods used by M. Hall. 
We invite the interested reader to show by elementary means that no finite group has exactly $15$ Sylow $7$-subgroups. 

\section*{Acknowledgment}
The author is supported by the German Research Foundation (project SA 2864/1-1).

\end{document}